\documentclass[12pt]{amsart}
\usepackage[shortlabels]{enumitem}
\usepackage{amsmath,amscd,amsthm,amsfonts,amssymb}
\usepackage{eucal}
\usepackage{mathrsfs}
\usepackage{float}
\usepackage{morefloats}
\usepackage[margin=1in]{geometry}
\usepackage{tikz-cd, comment, fancyvrb}
\usepackage{xcolor}

\RequirePackage{color}
\definecolor{myred}{rgb}{0.75,0,0}
\definecolor{mygreen}{rgb}{0,0.5,0}
\definecolor{myblue}{rgb}{0,0,0.65}

\usepackage{hyperref}
\hypersetup{citecolor=blue}
\usepackage{tikz}
\usetikzlibrary{matrix,arrows,decorations.pathmorphing}

\renewcommand{\bar}{\overline} 

\renewcommand{\tilde}{\widetilde}

\let\temp\emptyset
\let\emptyset\varnothing
\let\varnothing\temp 

\theoremstyle{plain}
\newtheorem{theorem}{Theorem}[section]

\newtheorem{lemma}[theorem]{Lemma}
\newtheorem*{lemma*}{Lemma}
\newtheorem*{proposition*}{Proposition}

\newtheorem*{truefact*}{Fact}

\theoremstyle{definition}

\theoremstyle{remark}
\newtheorem*{remark}{Remark}

\newcommand{\bb}[1]{\expandafter\newcommand\expandafter{\csname #1\endcsname}{{\mathbb {#1}}}} 

\bb C
\bb R
\bb Q
\bb Z
\bb N

\bb F

\newcommand{\mcH}{\mathcal H}

\renewcommand{\a}{\alpha}

\newcommand{\g}{\gamma}
\renewcommand{\d}{\delta}
\newcommand{\ep}{\varepsilon}

\newcommand{\la}{\lambda}

\newcommand{\ds}{\mathrm{d}s}
\newcommand{\dt}{\mathrm{d}t}
\newcommand{\dx}{\mathrm{d}x}
\newcommand{\dy}{\mathrm{d}y}

\newcommand{\du}{\mathrm{d}u}

\newcommand{\mrm}[1]{\expandafter\newcommand\expandafter{\csname #1\endcsname}{{\mathrm {#1}}}}

\mrm{Hom}
\mrm{Aut}
\mrm{End}
\mrm{Spec}
\mrm{Gal}

\mrm{PGL}
\mrm{PSL}
\mrm{GL}
\mrm{SL}
\mrm{PG}
\mrm{PSO}
\mrm{PS}
\mrm{PSU}

\mrm{lcm}

\newcommand{\p}{\mathfrak p}
\renewcommand{\S}{\mathfrak S}

\renewcommand{\mod}{\bmod}
\newcommand{\subsum}[1]{\sum_{\substack{#1}}}

\title[Restricted sums of singular series]{Sums of singular series along arithmetic progressions and with smooth weights}
\author{Vivian Kuperberg}
\thanks{The author is supported by NSF GRFP grant DGE-1656518 as well as the NSF Mathematical Sciences Research Program through the grant DMS-2202128, and would like to thank Kannan Soundararajan for many helpful comments and discussions.}

\begin{document}
\begin{abstract}
Sums of the singular series constants that appear in the Hardy--Littlewood $k$-tuples conjectures have long been studied in connection to the distribution of primes. We study constrained sums of singular series, where the sum is taken over sets whose elements are specified modulo $r$ or weighted by smooth functions. We show that the value of the sum is governed by incidences modulo $r$ of elements of the set in the case of arithmetic progressions and by pairings of the smooth functions in the case of weights. These sums shed light on sums of singular series in other formats.
\end{abstract}
\maketitle

\nocite{MR2290492}

\section{Introduction}

The Hardy--Littlewood $k$-tuples conjectures, and the constants known as singular series that appear within them, have long been studied in connection to the distribution of primes. These conjectures state that for any $k$-tuple $\mathcal H = \{h_1, \dots, h_k\}$ of distinct integers, the number of $k$-tuples of primes of the form $(n+h_1, \dots, n+h_k)$, with $n \le x$, is given by
\begin{equation}
\sum_{n \le x} \prod_{i=1}^k \Lambda(n+h_i) = \mathfrak S(\mathcal H)x + o(1)x,
\end{equation}
where $\mathfrak S(\mathcal H)$ is the singular series
\begin{equation}\label{eq:bg:singseries}
\mathfrak S(\mcH) = \prod_{p \text{ prime}} \frac{1-\nu_{\mcH}(p)/p}{(1-1/p)^k},
\end{equation}
and $\nu_{\mcH}(p)$ denotes the number of distinct residue classes modulo $p$ occupied by the elements of $\mcH$. 

In \cite{gallaghershortintervals}, Gallagher showed that the Hardy--Littlewood conjectures imply that the distribution of primes in intervals of size $h = \la \log x$ is Poissonian for fixed $\la$, by showing that the singular series is $1$ on average. In particular, he showed that for any fixed $k$,
\begin{equation*}
\subsum{\mathcal H \subset [1,h] \\ |\mathcal H| = k} \mathfrak S(\mathcal H) \sim \subsum{\mathcal H \subset [1,h] \\ |\mathcal H| = k} 1.
\end{equation*}

In 2004, Montgomery and Soundararajan \cite{MontgomerySoundararajanPrimesIntervals}
used a more refined estimate of sums of singular series to show that when $h$ is in a larger regime, with $h/\log x \to \infty$ but $h = o(x)$, the Hardy--Littlewood conjectures imply that the distribution of primes, counted with von Mangoldt weights, becomes Gaussian with mean $\sim h$ and variance $\sim h \log \frac xh$, matching numerical data. Instead of the singular series itself, they considered alternating sums of singular series, defining $\mathfrak S_0(\mathcal H) := \sum_{\mathcal J \subset \mathcal H} (-1)^{|\mathcal H \setminus \mathcal J|} \mathfrak S(\mathcal J).$ These sums have the effect of subtracting the main term from the outset, making it easier to understand lower-order contributions. The analogous Hardy--Littlewood conjecture says that
\begin{equation*}
\sum_{n \le x} \prod_{i=1}^k \left(\Lambda(n+h_i) - 1\right) = \mathfrak S_0(\mathcal H) x + o(x),
\end{equation*}
so that each term on the left-hand side has expectation $0$. 

Montgomery and Soundararajan \cite{MontgomerySoundararajanPrimesIntervals} showed that for fixed $k$, 
\begin{equation}\label{eq:intro:msthm2}
\subsum{\mathcal H \subset [1,h] \\ |\mathcal H| = k} \mathfrak S_0(\mathcal H) = \mu_k (-h \log h + A)^{k/2} + O_k(h^{k/2-1/(7k) + \ep}),
\end{equation}
where $\mu_k = 1 \cdot 3 \cdots (k-1)$ if $k$ is even, and $0$ if $k$ is odd, and $A = 2-\g_0 - \log 2\pi$, with $\g_0$ the Euler--Mascheroni constant. Their results depend crucially on work of Montgomery and Vaughan \cite{MontgomeryVaughanReducedResidues}
on the distribution of reduced residues mod $q$ in short intervals.

Here we develop analogs of the work of Montgomery and Vaughan as well as that of Montgomery and Soundararajan by studying sums of singular series with added conditions on the set $\mathcal H$. First, instead of summing over all subsets of $[1,h]$, we restrict to sets whose elements lie in arithmetic progressions. For a fixed modulus $r$ and congruence classes $c_1, \dots, c_k$ mod $r$, we study the sum
\begin{equation}\label{eq:mvarb:Rkhrcidef}
R_k(h;r,c_1,\dots,c_k) := \subsum{\mathcal H = \{h_1, \dots, h_k\} \subset [1,h] \\ |\mathcal H| = k \\ h_i \equiv c_i \mod{r}} \S_{0,r}(\mcH),
\end{equation}
where $\S_{0,r}(\mathcal H) = \sum_{\mathcal J \subset \mcH} \S_r(\mathcal J)(-1)^{|\mcH \setminus \mathcal J|}$ and $\S_r(\mcH)$ is the singular series of $\mcH$ away from $r$, given by
\begin{equation}\label{eq:bg:singseriesawayq}
\mathfrak S_r(\mcH) := \prod_{\substack{p \text{ prime} \\p \nmid r}} \frac{1-\nu_{\mcH}(p)/p}{(1-1/p)^k}.
\end{equation}

The study of these sums is of interest for two reasons. First, they appear in the work of Lemke Oliver and Soundararajan in \cite{lemkeoliversound} on bias in the distribution of consecutive primes in arithmetic progressions. Specifically, Lemke Oliver and Soundararajan conjecture that if $\pi(x;q,(a,b))$ is defined as the number of primes $p \le x$ with $p \equiv a \mod q$ and $p_{\text{next}} \equiv b \mod q$, then
\begin{equation*}
\pi(x;q,(a,b)) \sim \frac 1q \int_2^x \alpha(y)^{\epsilon_q(a,b)} \left(\frac{q}{\phi(q)\log y}\right)^2 \mathcal D(a,b;y)\dy,
\end{equation*}
where $\alpha(y) :- 1-\frac{q}{\phi(q)\log y}$, $\epsilon_q(a,b)$ is a constant defined only in terms of $q,a,$ and $b$, and
\begin{equation*}
\mathcal D(a,b;y) := \sum_{\substack{h > 0 \\ h \equiv b - a \mod q}} \sum_{\mathcal A \subset \{0,h\}} \sum_{\substack{\mathcal T \subset [1,h-1] \\ (t+a,q) = 1 \forall t\in \mathcal T}} (-1)^{|\mathcal T|} \mathfrak S_{0,q}(\mathcal A \cup \mathcal T)\left(\frac{q}{\phi(q)\alpha(y)\log y}\right)^{|\mathcal T|} \alpha(y)^{h\phi(q)/q}.
\end{equation*}
They proceed to heuristically estimate $\mathcal D(a,b;y)$, and thus $\pi(x;q,(a,b))$, by estimating a weighted version of $R_2(h;r,c_1,c_2)$. Our results generalize these arguments by estimating $R_k(h;r,c_1,c_2)$ for all $k$, which is necessary for understanding some of the error terms appearing in Lemke Oliver and Soundararajan's heuristic. Secondly, restricting sums of singular series in this manner may shed light on other questions about sums of singular series. We show in Theorem \ref{thm:mvarb:msthm2foraps} that the asymptotics for \eqref{eq:mvarb:Rkhrcidef} are governed by incidences among the $c_i$'s mod $r$. As discussed in \cite{KuperbergOddMoments}, we do not yet know the asymptotic average size of sums of $\mathfrak S_0(\mathcal H)$ when $|\mathcal H|$ is odd. The results of these more refined averages of singular series may clarify where the main term should be coming from for sums of singular series with an odd number of terms.

The sums over $\mathfrak S_{0,r}$ admit the expansion
\begin{equation*}
\subsum{1 \le h_1, \dots, h_k \le h \\ h_i \equiv c_i \mod r \: \forall i \\\text{distinct}} \mathfrak S_{0,r}(\{h_1, \dots, h_k\})= \subsum{q_1, \dots, q_k \\ q_i > 1 \\ (q_i,r) = 1} \frac{\mu(q_i)}{\phi(q_i)} \subsum{a_1, \dots, a_k \\ 1\le a_i \le q_i \\ (a_i,q_i) = 1 \\ \sum_i a_i/q_i \in \Z} \subsum{1 \le h_1, \dots, h_k \le h \\ h_i \equiv c_i \mod r \text{distinct}}\prod_{i=1}^k e\left(\frac{a_ih_i}{q_i}\right).
\end{equation*}
Following \cite{MontgomerySoundararajanPrimesIntervals} and \cite{MontgomeryVaughanReducedResidues}, we first consider a related quantity, where the summands $q_i$ are restricted to divide a secondary modulus $q > 1$ and the $h_i$'s are not necessarily distinct, to get
\begin{equation}\label{eq:mvarb:vkdefwmod}
V_k(q,h;r,c_1, \dots, c_k) :=
\subsum{q_1, \dots, q_k \\ 1 < q_i|q \\ (q_i,r) = 1} \frac{\mu(q_i)}{\phi(q_i)} \subsum{a_1, \dots, a_k \\ 1\le a_i \le q_i \\ (a_i,q_i) = 1 \\ \sum_i a_i/q_i \in \Z} \prod_{i=1}^k E_{r,c_i}\left(\frac{a_ih_i}{q_i}\right),
\end{equation}
where for $\a \in \R$,
\begin{equation}\label{eq:Ercia}
E_{r,c_i}(\a) := \subsum{m \le h \\ m \equiv c_i \mod r} e(m\a).
\end{equation}

In order to state our result, we will need to fix some notation concerning perfect matchings of $[1,k]$. For $k \ge 1$, a \emph{perfect matching} $\sigma$ of $[1,k]$ is a set $\sigma = \{(i,j)\}$ of unordered pairs in $[1,k]$ so that each element is paired with exactly one other element, i.e. each $i$ appears in exactly one pair $(i,j)$, and $i \ne j$. Since each pair $(i,j) \in \sigma$ is unordered, we will generally choose to write the representative with $i < j$, so that $\sigma \{(i,j)\}$ with $i < j$. Let $\mathcal B_k$ denote the set of perfect matchings of $[1,k]$, so that
\begin{equation}\label{eq:mvarb:matchingnotation}
\mathcal B_k := \left\{ \sigma = \{(i,j)\} : \begin{array}{l} \forall (i,j) \in \sigma, 1 \le i < j \le k \\ \forall i \in [1,k], \exists! j \in [1,k] \text{ with } (i,j) \in \sigma \text{ or } (j,i) \in \sigma \end{array} \right\}.
\end{equation}
Note that when $k$ is odd, $\mathcal B_k = \emptyset$. Moreover, for a set of integers $\{a_1, \dots, a_k\}$, we will denote by $\mathcal B(a_1, \dots, a_k)$ the set of matchings of $\{a_1, \dots, a_k\}$ into pairs, so that $\mathcal B_k = \mathcal B([1,k])$. 

In Section \ref{sec:pfofmsthm1}, we prove the following result, which mirrors Theorem 1 of \cite{MontgomerySoundararajanPrimesIntervals}.
\begin{theorem}\label{thm:mvarb:msthm1foraps}
Fix a modulus $r \ge 1$, an integer $k \ge 1$, and $k$ congruence classes $c_1, \dots, c_k$ modulo $r$. Define $\mathcal B_k$ as in \eqref{eq:mvarb:matchingnotation}. Let $q \ge 1$ be a squarefree integer with $(r,q) = 1$, and define $V_k(q,h;r,c_1,\dots,c_k)$ as in \eqref{eq:mvarb:vkdefwmod}. For $h \ge 3$,
\begin{equation*}
V_k(q,h;r,c_1,\dots, c_k) = \subsum{\sigma \in  \mathcal B_k} \prod_{(i,j) \in \sigma} V_2(q,h;r,c_i,c_j) + O_{r,k}\left(h^{k/2-1/(7k)}\left(\frac{q}{\phi(q)}\right)^{2^k + k/2}\right).
\end{equation*}
\end{theorem} 

In order to state our main result on the asymptotics of $R_k(h;r,c_1,\dots,c_k)$, we define some further notation.
For $1 \le \ell \le r$, define
\begin{equation*}
\mathcal C_{\ell} := \left\{i: c_i \equiv \ell \mod r\right\}.
\end{equation*}
Note that some of the sets $\mathcal C_{\ell}$ may be empty, and that $\bigcup_{\ell = 1}^r \mathcal C_{\ell} = [1,k]$.
We will say that a partition $\mathcal P = \{\mathcal S_1, \dots, \mathcal S_M\}$ of $[1,k]$ \emph{refines} $\{\mathcal C_\ell\}_{\ell \in [1,k]}$ if for each $\mathcal S_m \in \mathcal P$, there exists some $\ell$ with $\mathcal S_m \subset \mathcal C_\ell$; note that $\ell$ is then unique. For such a partition, write $\mathcal P \preceq \{\mathcal C_\ell\}_{\ell \in [1,k]}$ and define $c(\mathcal S_m)$ to be the value $\ell$ with $\mathcal S_m \subset \mathcal C_\ell$.

\begin{theorem}\label{thm:mvarb:msthm2foraps}
Fix a modulus $r \ge 1$ and an integer $k \ge 1$, as well as $k$ congruence classes $c_1, \dots, c_k$ modulo $r$. Define $R_k(h;r,c_1, \dots, c_k)$ by \eqref{eq:mvarb:Rkhrcidef}. Then for $h \ge 3$, 
\begin{align}
R_k(h;r,c_1,\dots, c_k) = &\sum_{0 \le j \le k/2} (-1)^j \subsum{\mathcal P \preceq \{\mathcal C_\ell\}_{\ell \in [1,k]} \\ \mathcal P = \{\mathcal S_1, \dots, \mathcal S_{k-j}\} \\ |\mathcal S_m| = 2 \:\forall 1 \le m \le j \\ |\mathcal S_m| = 1 \: \forall j < m \le k-j} \left(\frac hr \subsum{d|Q \\ d>1} \frac{\mu(d)^2}{\phi(d)}\right)^j \nonumber\\
&\sum_{\sigma \in \mathcal B(j+1, \dots, k-j)} \prod_{(i_1,i_2) \in \sigma} V_2(Q,h; r,c(\mathcal S_{i_1}), c(\mathcal S_{i_2})), + O_{r,k}(h^{k/2-1/(7k) + \ep}). 
\end{align}

In particular, if $\#\tilde{\mathcal B}(c_1, \dots, c_k)$ is the number of ways to pair the $c_i$'s such that every pair has equal values, then 
\begin{equation*}
R_k(h;r,c_1, \dots, c_k) = \#\tilde{\mathcal B}(c_1, \dots, c_k) \left( - h \frac{\phi(r)}{r} \log h  + C_0(r)h\right)^{k/2} + O_{r,k}(h^{k/2}(\log h)^{k/2-1}).
\end{equation*}
\end{theorem}

\begin{remark}
When all the $c_i$'s are congruent mod $r$, then $\#\tilde{\mathcal B}(c_1, \dots, c_k) = \mu_k$ and this theorem implies Theorem 2 of \cite{MontgomerySoundararajanPrimesIntervals}. When the $c_i$'s are not all congruent, the estimate depends crucially on the precise arrangement of the $c_i$'s. Nevertheless, while the value of the main term is dependent on the $c_i$'s, we always have the upper bound $R_k(h;r,c_1, \dots, c_k) \ll_{r,k} h^{k/2}(\log h)^{k/2}$.
\end{remark}

The framework of Theorems \ref{thm:mvarb:msthm1foraps} and \ref{thm:mvarb:msthm2foraps} also applies to sums of singular series weighted by smooth functions. Let $f_1, \dots, f_k:\R){\ge 0} \to \C$ be functions with compact supports contained in $(0,\infty)$ and such that $|\hat{f_i}(\xi)|\ll \xi^{-2}$, where the Fourier transform $\hat f$ is defined as
\[\hat f(\xi) := -\int_{\infty}^\infty f(x) e^{-2\pi i x \xi} \dx.\]
For $h \in \N$, we are interested in the quantity
\begin{equation}\label{eq:Rkh-smooth-def}
R_k(h;f_1,\dots,f_k) := \subsum{h_1, \dots, h_k \in \Z \\ \text{distinct}}\prod_{i=1}^k f_i\left(\frac{h_i}{h}\right) \mathfrak S_0(\{h_1, \dots, h_k\}).
\end{equation}
The size of $R_k(h;f_1,\dots,f_k)$ as $h \to \infty$ is similar in structure to the analogous result for sums of singular series along arithmetic progressions: the main term is a sum over perfect matchings of $[1,k]$, and for each matching $\sigma$ the contribution is determined by interactions of $f_i$ and $f_j$ for each pair $(i,j)\in \sigma$. 

We set some notation before stating our results.
Define
\begin{equation}\label{eq:Vqhfi-def}
V(q,h;f_1, \dots, f_k) := \sum_{1 < q_1, \dots, q_k |q} \prod_{i=1}^k \frac{\mu(q_i)}{\phi(q_i)} \subsum{a_1, \dots, a_k \\ 1 \le a_i \le q_i \\ (a_i,q_i) = 1 \\ \sum_i a_i/q_i \in \Z} E_{f_1,h}\left(\frac{a_1}{q_1}\right) \cdots E_{f_k,h}\left(\frac{a_k}{q_k}\right),
\end{equation}
where for $\a \in \R$ and $f$ smooth, with compact support, and such that $|\hat{f}(\xi)| = O(|\xi|^{-2})$,
\begin{equation}\label{eq:mvaw:defEaphisum}
E_{f,h}(\a) := \sum_{m=-\infty}^\infty f\left(\frac mh\right) e(m\a).
\end{equation}

Since $f$ is smooth, $E_{f,h}(\a)$ is a smoother indicator function of values of $\a$ that are close to $0$. In particular, by Poisson summation,
\begin{equation*}
E_{f,h}(\a) = h \sum_{n=-\infty}^\infty \hat{f}(h(n-\a)).
\end{equation*}
By assumption, $\hat{f}(\xi) = O(|\xi|^{-2})$. For any real number $\a$, only one value of $\a-n$ will be in the interval $[-1/2,1/2)$; let $\bar{\a}$ denote this value, so that $\bar{\a}$ is the representative of $\a \mod 1$ satisfying $-1/2 \le \bar{\a} < 1/2$. Then
\begin{equation}\label{eq:Efh_fourier}
E_{f,h}(\a) = h \hat{f}(-h\bar{\a}) + h \sum_{n=1}^\infty O((hn)^{-2})  = h \hat{f}(-h\bar{\a}) + O(h^{-1}).
\end{equation}

The following results about $V_k(q,h;f_1,\dots,f_k)$ hold, analogous to Theorem \ref{thm:mvarb:msthm1foraps} and Lemma \ref{lem:mvarb:V2comp}.

\begin{theorem}\label{thm:msthm1-for-smooth}
Fix $k \ge 1$ and $k$ smooth functions $f_1, \dots, f_k:\R_{\ge 0} \to \C$ with compact supports $\mathrm{supp}(f_i) \subset (0,\infty)$ and such that $|\hat f_i(\xi)| = O(|\xi|^{-2})$ for all $1\le i\le k$. Define $\mathcal B_k$ as in \eqref{eq:mvarb:matchingnotation}. Let $q \ge 1$ be a squarefree integer, and define $V_k(q,h;f_1,\dots,f_k)$ as in \eqref{eq:Vqhfi-def}. For $h \ge 3$,
\begin{equation*}
V_k(q,h;f_1,\dots, f_k) = \subsum{\sigma \in  \mathcal B_k} \prod_{(i,j) \in \sigma} V_2(q,h;f_i,f_j) + O_{f_1,\dots,f_k}\left(h^{k/2-1/(7k)}\left(\frac{q}{\phi(q)}\right)^{2^k + k/2}\right).
\end{equation*}
\end{theorem}

\begin{lemma}\label{lem:V2-smooth}
Fix $h \ge 1$ and let $f_1,f_2:\R_{\ge 0} \to \C$ be smooth functions with compact supports $\mathrm{supp}(f_i) \subset (0,\infty)$ such that $|\hat{f_i}(\xi)| \ll O(|\xi|^{-2})$. Define $Q := \prod_{p \le h^2} p$, and define $V_2(q,h;f_1,f_2)$ via \eqref{eq:Vqhfi-def}. Then
\begin{equation}
V_2(Q,h;f_1,f_2) = (-\hat{f_1}(0)\hat{f_2}(0) + \{\mathcal M f\}(2))h^2 - \frac{\{\mathcal M f'\}(1)}{2}h \log h + O_{f_1,f_2}(h),
\end{equation}
where $\{\mathcal M f\}(s)$ is the Mellin transform of $f$, defined by
\begin{equation}\label{eq:mellin-transform-definition}
\{\mathcal M f\}(s) := \int_0^\infty x^{s-1} f(x) \dx.
\end{equation}
\end{lemma}

Using precisely the same techniques as in the case of arithmetic progressions, an analog to Theorem \ref{thm:mvarb:msthm2foraps} also holds for the smooth functions setting, where the input to each summand is clarified by Lemma \ref{lem:V2-smooth}.

\begin{theorem}\label{thm:msthm2-for-smooth}
Fix an integer $k$ and smooth functions $f_1,\dots,f_k:\R_{\ge 0} \to \C$ with compact support such that $\mathrm{supp}(f_i) \subset (0,\infty)$ and such that $|\hat{f_i}(\xi)| = O(|\xi|^{-2})$ for each $1 \le i \le k$. Define $R_k(h;f_1,\dots,f_k)$ via \eqref{eq:Rkh-smooth-def}. Then for $h \ge 3$,
\begin{align}
R_k(h;f_1,\dots, f_k) = &\sum_{0 \le j \le k/2} (-1)^j \subsum{\mathcal P = \{\mathcal S_1, \dots, \mathcal S_{k-j}\} \\ |\mathcal S_m| = 2 \:\forall 1\le m \le j \\ |\mathcal S_m| = 1 \: \forall j < m \le k-j} \left(h \subsum{d|Q \\ d>1} \frac{\mu(d)^2}{\phi(d)}\right)^j \nonumber\\
&\sum_{\sigma \in \mathcal B(j+1, \dots, k-j)} \prod_{(i_1,i_2) \in \sigma} V_2(Q,h; f_{\mathcal S_{i_1}}, f_{\mathcal S_{i_2}}), + O_{r,k}(h^{k/2-1/(7k) + \ep}),
\end{align}
where the sum is taken over partitions of $[1,k]$ where each part has either $1$ or $2$ elements, and for $|\mathcal S_m| = 1$, $f_{\mathcal S_m}$ denotes $f_j$ where $j \in \mathcal S_m$.
\end{theorem}

The proofs of Theorems \ref{thm:msthm1-for-smooth} and \ref{thm:msthm2-for-smooth} are identical to the proofs of Theorems \ref{thm:mvarb:msthm1foraps} and \ref{thm:mvarb:msthm2foraps}, so we omit them. However, Theorems \ref{thm:msthm1-for-smooth} and \ref{thm:mvarb:msthm1foraps}, which provide asymptotics for $V_k$ in the cases of smooth weighting and arithmetic progressions, rely on lemmas about sums of $E_{f,h}(\alpha)$. The results are fundamentally the same, but the flavor of several of these lemmas is somewhat different in the smooth case, so we record them in Section \ref{sec:smooth-functions}. We also provide the proof of Lemma \ref{lem:V2-smooth}, which is similar to the proof of Lemma \ref{lem:mvarb:V2comp}.

The organization of this paper is as follows. In Section \ref{sec:pfofmsthm1}, we prove Theorem \ref{thm:mvarb:msthm1foraps}. In Section \ref{sec:twoterms}, we compute certain sums of 2-term singular series which will then be used in the proof of Theorem \ref{thm:mvarb:msthm2foraps}. In Section \ref{sec:mvarb:msthm2pf}, we prove Theorem \ref{thm:mvarb:msthm2foraps}. Finally, in Section \ref{sec:smooth-functions} we discuss the smooth case.

\section{Proof of Theorem \ref{thm:mvarb:msthm1foraps}}\label{sec:pfofmsthm1}

The arguments in this section closely follow the arguments in the proof of Lemma 8 in \cite{MontgomeryVaughanReducedResidues} and Theorem 1 in \cite{MontgomerySoundararajanPrimesIntervals}. Here we outline the argument; where it differs, we present detailed explanations, but many steps are cited to \cite{MontgomeryVaughanReducedResidues} and \cite{MontgomerySoundararajanPrimesIntervals}.

The functions $E_{r,c_i}(\a)$, defined in \eqref{eq:Ercia}, are modular versions of sums $E(\a):= \sum_{m\le h} e(m\a)$. The sums $E(\a)$ are large if $\a$ is close to an integer and otherwise exhibit large amounts of cancellation. Since $E_{r,c_i}(\a)$ is summed only over $m$ in a set congruence class modulo $r$, it will also take large values when $\a$ is close to a multiple of $r$. For any $c,r$, write
\begin{align*}
E_{r,c}(\a) &= \subsum{m \le h} \frac 1r \subsum{n = 1}^r e\left(\frac{(m-c)n}{r}\right) e(m\a) \\
&= \frac 1r \subsum{n=1}^r  e(-cn/r) \sum_{m \le h} e(m\a + mn/r),
\end{align*}
so that
\begin{equation*}
|E_{r,c}(\a)|\le \frac 1r \sum_{n=1}^r \min\left\{h, \frac{1}{\|\a + n/r\|}\right\}. 
\end{equation*}
Define 
\begin{equation}\label{eq:mvarb:Fralphadef}
F_{r}(\a) := \frac 1r \sum_{n=1}^r \min\left\{h, \frac{1}{\|\a + n/r\|}\right\},
\end{equation} 
so that $|E_{r,c}(\a)| \le F_{r}(\a).$
We can then closely follow the analysis of Montgomery and Vaughan in \cite{MontgomeryVaughanReducedResidues} and Montgomery and Soundararajan in \cite{MontgomerySoundararajanPrimesIntervals}.
In particular, $F_r(\a)$ can be bounded, up to some dependence on $r$, by the same bounds as appear in Lemmas 4,5,6, and 7 of \cite{MontgomeryVaughanReducedResidues} and using the same arguments. We arrive at the following result.

\begin{theorem}\label{thm:mvarb:mvoddterms}
Fix $r,k,q \ge 1$. Let $\mathcal S(q)$ denote the set of $k$-tuples $(q_1, \dots, q_k)$ such that for all $i$, $q_i|q$, $q_i > 1$, and the $q_i$'s are not equal in pairs and otherwise distinct; that is to say, there is no reordering permutation $\sigma$ of the $k$ indices with $q_{\sigma(i)} = q_{\sigma(i+k/2)}$ for all $1 \le i \le k/2$, and no other equalities. Then
\begin{align*}
\subsum{(q_1, \dots, q_k) \in \mathcal S(q)} \prod_{i=1}^k \frac{\mu(q_i)^2}{\phi(q_i)} \subsum{a_1, \dots, a_k \\ 1\le a_i \le q_i  \\ (a_i,q_i) = 1 \\ \sum_i a_i/q_i \in \Z} \prod_{i=1}^k F_{r}\left(\frac{a_i}{q_i}\right) \ll_r q h^{k/2-1/7k} \left(\frac{\phi(q)}{q}\right)^{k/2-2^k}.
\end{align*}
\end{theorem}

Theorem \ref{thm:mvarb:mvoddterms} accounts for all terms in $V_k(q,h;r,c_1,\dots, c_k)$ except for those terms where the $q_i$'s are equal in pairs and otherwise distinct. When $k$ is odd, there are no such terms, so assume that $k$ is even. By the same argument as in \cite{MontgomerySoundararajanPrimesIntervals}, we can drop the assumption that the pairs are distinct, so that the terms left to be estimated are given by
\begin{equation*}
\subsum{\sigma \in \mathcal B_k} \subsum{(q_i)_{(i,j) \in \mathcal B_k} \\ 1 < q_i|q} \prod_{(i,j) \in \mathcal B_k} \frac{\mu(q_i)^2}{\phi(q_i)^2} \subsum{(b_i)_{(i,j) \in \mathcal B_k}\\1 \le b_i \le q_i \\ \sum_i b_i/q_i \in \Z} \prod_{(i,j) \in \mathcal B_k} J_{r,c_i,c_j}(b_i,q_i),
\end{equation*}
where for each $i$,
\begin{equation*}
J_{r,c_i,c_j}(b_i,q_i) := \subsum{1 \le a_i \le q_i \\ (a_i,q_i) = 1 \\ (b_i-a_i,q_i) = 1} E_{r,c_i}\left(\frac{a_i}{q_i}\right)E_{r,c_j}\left(\frac{b_i-a_i}{q_i}\right).
\end{equation*}

Each term $\sigma$ in our sum is identical up to changing labels, so without loss of generality we will work with the term $\sigma = \{(i,k/2 + i) : 1 \le i \le k/2\}$, so that $q_i = q_{i + k/2}$ for all $1 \le i \le k/2$. This term is given by
\begin{equation}\label{eq:mvarb:msthm1:fixedsigmasum}
\subsum{q_1, \dots, q_{k/2} \\ 1 < q_i|q} \prod_{i=1}^{k/2} \frac{\mu(q_i)^2}{\phi(q_i)^2} \subsum{b_1, \dots, b_{k/2} \\ 1 \le b_i \le q_i \\ \sum_i b_i/q_i \in \Z} \prod_{i = 1}^{k/2} J_{r,c_i,c_{k/2+i}}(b_i,q_i).
\end{equation}
Let $\mathcal J \subset [1,k/2]$ be the subset of $i$'s with $0 < b_i < q_i$ instead of $b_i = q_i$. For $i \not\in \mathcal J$, we have
\begin{equation*}
J_{i,k/2+i}(q_i,q_i) = \subsum{1 \le a_i \le q_i \\ (a_i,q_i) = 1} E_{r,c_i}\left(\frac{a_i}{q_i}\right) E_{r,c_{k/2 +i}} \left(-\frac{a_i}{q_i}\right),
\end{equation*}
which, when summed over $1 < q_i|q$ with the weight $\frac{\mu(q_i)^2}{\phi(q_i)^2}$, is equal to $V_2(q,h; r,c_i,c_{k/2+i})$. Thus, \eqref{eq:mvarb:msthm1:fixedsigmasum} is equal to
\begin{equation}\label{eq:mvarb:msthm1:VsandWs}
\subsum{\mathcal J \subset [1,k/2]} \prod_{i \not\in\mathcal J} V_2(q,h;r,c_i,c_{k/2+i}) W_{\mathcal J}(q,h;r,(b_i)_{i\in\mathcal J}),
\end{equation}
where 
\begin{equation*}
W_{\mathcal J}(q,h;r,(b_i)_{i\in\mathcal J}) := \subsum{(q_i)_{i \in \mathcal J} \\ 1 < q_i|q} \prod_{i \in \mathcal J} \frac{\mu(q_i)^2}{\phi(q_i)^2} \subsum{(b_i)_{i \in\mathcal J} \\ 0 < b_i < q_i \\ \sum_{i \in \mathcal J} b_i/q_i \in \Z} \prod_{i \in \mathcal J} J_{c_i,c_{k/2+i}}(b_i,q_i).
\end{equation*}

The term $\mathcal J = \emptyset$ gives the desired main term, so it remains to show that the terms with $\mathcal J \ne \emptyset$ are smaller. By following the reasoning on page 11 of \cite{MontgomerySoundararajanPrimesIntervals}, we get that
\begin{equation*}
W_{\mathcal J}(q,h;r,(b_i)_{i\in\mathcal J}) \ll_r h^{|\mathcal J|/2} \left(\frac q{\phi(q)}\right)^{2^{|\mathcal J|}},
\end{equation*}
and that for any $c_1,c_2 \mod r$,
\begin{equation*}
V_2(q,h;r,c_1,c_2) \ll_r h\frac{q}{\phi(q)}.
\end{equation*}
Applying these estimates to \eqref{eq:mvarb:msthm1:VsandWs} completes the proof of Theorem \ref{thm:mvarb:msthm1foraps}.

\section{Auxiliary lemmas: two-term computations}\label{sec:twoterms}

In order to prove Theorem \ref{thm:mvarb:msthm2foraps}, we will ultimately invoke Theorem \ref{thm:mvarb:msthm1foraps}, which will then relate quantities involving $k$-term singular series to quantities involving $2$-term singular series. In this section, we state two lemmas which compute, respectively, sums of two-term singular series in arithmetic progressions, and the quantity $V_2(Q,h;r,c_1,c_2)$ when $Q$ is the product of all primes below $h^2$, which will be applied in Section \ref{sec:mvarb:msthm2pf}.

The following lemma is a computation of sums of two-term singular series to the modulus $q$. Its proof is nearly identical to the proof of Proposition 2.1 in \cite{lemkeoliversound}, so we omit it. Similar quantities were previously studied in \cite{MR1073669} and \cite{MR1954710}.
\begin{lemma}\label{lem:mvarb:twotermSrcomp}
Fix $r, h \ge 1$, and let $v \mod r$ be any residue class. Define
\begin{equation*}
S(r,v,h) := \subsum{1 \le m \le h \\ m \equiv v \mod r} (h-m) \mathfrak S_r(\{0,m\}).
\end{equation*}
Then when $v = 0$,
\begin{equation*}
S(r,v,h) = \frac{h^2}{2r} - \frac{h}{2} \phi(r) \left(\log \frac hr + \log 2\pi + \g_0 + \sum_{p|r} \frac{\log p}{p-1}\right) + O_r(h^{1/2 + \ep}),
\end{equation*}
where $\g_0$ denotes the Euler--Mascheroni constant. Meanwhile, if $v \ne 0$, then if $d = (v,r)$, 
\begin{align*}
S(r,v,h) = \frac{h^2}{2r} &- \frac h2 \frac{\phi(r)}{r} \frac 1{\phi(r/d)} \Lambda(r/d) \nonumber \\
&+ \frac h{\phi(r/d)}\sum_{\chi \ne \chi_0 \mod{r/d}} \bar{\chi}(v/d) L(0,\chi) L(1,\chi) A_{r,\chi} + O_r(h^{1/2 + \ep}),
\end{align*}
with
\begin{equation}\label{eq:mvarb:defArchi}
A_{r,\chi} = \prod_{p|r} \left(1-\frac{\chi(p)}{p}\right) \prod_{p\nmid r} \left(1 - \frac{(1-\chi(p))^2}{(p-1)^2}\right).
\end{equation}
\end{lemma}

\begin{lemma}\label{lem:mvarb:V2comp}
Fix integers $r,h \ge 1$ and two congruence classes $c_1, c_2 \mod r$. Define 
\begin{equation*}
Q = \prod_{\substack{p \le h^2 \\ p \nmid r}} p,
\end{equation*}
so that $Q$ is the product of all primes below $h^2$ that do not divide $r$. Let $V_2(Q,h; r,c_1,c_2)$ be defined as in \eqref{eq:mvarb:vkdefwmod}. Then if $c_1 \equiv c_2 \mod r$,
\begin{equation*}
V_2(Q,h; r,c_1,c_1) = \frac hr \subsum{d|Q\\d>1} \frac{\mu(d)^2}{\phi(d)} - h \frac{\phi(r)}{r} \log h  + C_0(r)h + O_r(h^{1/2 + \ep}),
\end{equation*}
where 
\begin{equation}\label{eq:mvarb:C0rdef}
C_0(r) = \frac{\phi(r)}{r} \left(\log \frac{r}{2\pi} -\g_0 - \sum_{p|r}\frac{\log p}{p-1}\right),
\end{equation}
and $\g_0$ is the Euler--Mascheroni constant. If $c_1 \not \equiv c_2 \mod r$, define $d = (c_1-c_2,r)$, with $d < r$. Then 
\begin{align}
V_2(Q,h; r,c_1,c_2) = -&h \frac{\phi(r)}{r^2\phi(r/d)} \Lambda(r/d) \nonumber \\
&+ \frac{2h}{r\phi(r/d)} \sum_{\chi \ne \chi_0 \mod{r/d}} \bar{\chi}\left(\frac{c_1-c_2}{d}\right) L(0,\chi) L(1,\chi) A_{r,\chi} + O_r(h^{1/2 + \ep}),
\end{align}
where $A_{r,\chi}$ is defined in \eqref{eq:mvarb:defArchi}.
\end{lemma}
\begin{proof}
Begin by expanding
\begin{align*}
V_2(Q,h; r,c_1,c_2) &= -E_{r,c_1}(1)E_{r,c_2}(1) + \sum_{q|Q} \frac{\mu(q)^2}{\phi(q)^2} \subsum{1 \le a \le q \\ (a,q) = 1} E_{r,c_1}\left(\frac aq\right)E_{r,c_2}\left(-\frac aq\right) \\
&= - \frac{h^2}{r^2} + \sum_{q|Q} \frac{\mu(q)^2}{\phi(q)^2}\subsum{1 \le a \le q \\ (a,q) = 1} \subsum{m_1,m_2 \le h \\ m_i \equiv c_i \mod r} e\left((m_1-m_2)\frac aq\right).
\end{align*}

Assume first that $c_1 \equiv c_2 \mod r$. Then
\begin{equation}\label{eq:V2firstexpansion}
V_2(Q,h; r,c_1,c_1) = -\frac{h^2}{r^2} + \sum_{q|Q} \frac{\mu(q)^2}{\phi(q)^2} \subsum{1 \le a \le q \\ (a,q) = 1} \subsum{|m| \le h \\ r|m} \frac 1r (h-|m|) e\left(m\frac aq\right).
\end{equation}
Let $\mathbf c_q(m)$ denote the Ramanujan sum $\subsum{1 \le a \le q \\ (a,q) = 1} e\left(m \frac aq \right)$. The expression \eqref{eq:V2firstexpansion} is then
\begin{align*}
V_2(Q,h; r,c_1,c_1) &= -\frac{h^2}{r^2} + \frac 1r \sum_{q|Q} \frac{\mu(q)^2}{\phi(q)^2}\left(h \mathbf c_q(0) + 2 \sum_{1 \le m \le h/r} (h-rm) \mathbf c_q(rm)\right) \\
&= -\frac{h^2}{r^2} + \frac hr \sum_{q|Q} \frac{\mu(q)^2}{\phi(q)^2} + \frac 2r \sum_{1 \le m \le h/r} (h-rm) \sum_{q|Q} \frac{\mu(q)^2}{\phi(q)^2} \mathbf c_q(rm).
\end{align*}
The inside sum over $q|Q$ is multiplicative and, since $Q$ is the product of primes $p\le h^2$ not dividing $r$, it is given by
\begin{equation*}
\sum_{q|Q} \frac{\mu(q)^2}{\phi(q)^2} \mathbf c_q(rm) = \prod_{\substack{p|m \\ p \nmid r}} \left(1+\frac 1{p-1}\right) \prod_{\substack{p \nmid m \\ p \nmid r}} \left(1-\frac 1{(p-1)^2}\right) + O_r(h^{-2}) = \mathfrak S_r(\{0,m\}) + O_r(h^{-2}),
\end{equation*}
by the definition of $\mathfrak S_r$ in \eqref{eq:bg:singseriesawayq}.

Then
\begin{equation*}
V_2(Q,h; r,c_1,c_1) = -\frac{h^2}{r^2} + \frac hr \sum_{q|Q} \frac{\mu(q)^2}{\phi(q)^2} + \frac 2r \sum_{1 \le m \le h/r} (h-rm) \mathfrak S_r(\{0,m\}) + O_r(1),
\end{equation*}
which by Lemma \ref{lem:mvarb:twotermSrcomp} is equal to
\begin{equation*}
-\frac{h^2}{r^2} + \frac hr \sum_{q|Q} \frac{\mu(q)^2}{\phi(q)^2} + \frac{h^2}{r^2} - h \frac{\phi(r)}{r} \left(\log \frac hr + \log 2\pi + \g_0 + \sum_{p|r} \frac{\log p}{p-1}\right) + O_r(h^{1/2 + \ep}).
\end{equation*}
After a little rearranging this gives the desired result.

Now assume $c_1 \not\equiv c_2 \mod r$. In this case,
\begin{align*}
V_2(Q,h;r,c_1,c_2) &= -\frac{h^2}{r^2} + \sum_{q|Q} \frac{\mu(q)^2}{\phi(q)^2} \subsum{1 \le a \le q \\ (a,q) = 1} \subsum{|m| \le h \\ m \equiv c_1 - c_2 \mod r} \frac 1r (h-|m|) e\left(m\frac aq\right) \\
&= -\frac{h^2}{r^2} + \subsum{|m| \le h \\ m \equiv c_1 - c_2 \mod r} \frac 1r (h-|m|) \sum_{q|Q} \frac{\mu(q)^2}{\phi(q)^2} \mathbf c_q(m) \\
&= -\frac{h^2}{r^2} + \subsum{|m| \le h \\ m \equiv c_1 - c_2 \mod r} \frac 1r (h-|m|) \mathfrak S_r (\{0,m\}) + O_r(1),
\end{align*}
where if $m \equiv c_1 - c_2 \mod r$, then $m \ne 0$. 

Applying Lemma \ref{lem:mvarb:twotermSrcomp} completes the proof.
\end{proof}

\section{Proof of Theorem \ref{thm:mvarb:msthm2foraps}}\label{sec:mvarb:msthm2pf}

Throughout this section, fix $r, k,h \ge 1$ and set $Q = \prod_{\substack{p \le y \\ p \nmid r}} p$, where $y = h^{k+1}$. 

Begin with the expansion
\begin{equation}\label{eq:mvarb:rkstartingpoint}
R_k(h;r,c_1, \dots, c_k) = \subsum{q_1, \dots, q_k \\ 1 < q_i|Q} \prod_{i=1}^k \frac{\mu(q_i)}{\phi(q_i)} \subsum{a_1, \dots, a_k \\ 1 \le a_i \le q_i \\ (a_i,q_i) = 1 \\ \sum_i a_i/q_i \in \Z} \subsum{d_1, \dots, d_k \\ 1 \le d_i \le h \\ \text{distinct} \\ d_i \equiv c_i \mod r}  e\left(\sum_{i=1}^k \frac{a_id_i}{q_i}\right) + O_r(1),
\end{equation}
where the error term is due to our choice of $Q$. The expression on the right-hand side of \eqref{eq:mvarb:rkstartingpoint} is very close to $V_k(Q,h;r,c_1, \dots, c_k)$, but in order to apply Theorem \ref{thm:mvarb:msthm1foraps}, we will need to remove the distinctness condition on the $d_i$'s. As in the proof of Theorem 2 from \cite{MontgomerySoundararajanPrimesIntervals}, removing this condition will be the bulk of our work.

This distinctness condition is heavily dependent on the congruence classes $c_1, \dots, c_k$; in particular, if $c_i \not\equiv c_j \mod r$, then $d_i$ and $d_j$ never coincide and the distinctness condition is immaterial. Our arguments follow those of \cite{MontgomerySoundararajanPrimesIntervals} closely, but with additional bookkeeping in order to account for the congruence classes $c_1, \dots, c_k \mod r$. 

For a given tuple $(d_1, \dots, d_k)$ with $1 \le d_i \le h$ and $d_i \equiv c_i \mod r$ for all $i$, put $\d_{ij} = 1$ if $d_i = d_j$ and $\d_{ij} = 0$ otherwise. Then
\begin{equation*}
\prod_{\ell = 1}^r \prod_{\substack{i,j \in \mathcal C_{\ell} \\ i < j}} (1-\d_{ij}) = \begin{cases} 1 &\text{ if the } d_i\text{ are distinct } \\ 0 &\text{ otherwise.}\end{cases}
\end{equation*}
When the left-hand side above is expanded, it is a linear combination of products of the $\d$ symbols. Let $\Delta$ denote one such product, and let $|\Delta|$ denote the number of $\d_{ij}$ in the product. As in \cite{MontgomerySoundararajanPrimesIntervals}, define an equivalence relation on these $\d$-products by setting $\Delta_1 \sim \Delta_2$ if $\Delta_1$ and $\Delta_2$ have the same value for all choices of $d_i$'s; for example, $\d_{12}\d_{23} \sim \d_{12}\d_{13} \sim \d_{12}\d_{23}\d_{13}.$

Recall that a partition $\mathcal P = \{\mathcal S_1, \dots, \mathcal S_M\}$ of $[1,k]$ \emph{refines} $\{\mathcal C_\ell\}_{\ell \in [1,k]}$ if for each $\mathcal S_m \in \mathcal P$, there exists some $\ell$ with $\mathcal S_m \subset \mathcal C_\ell$; note that $\ell$ is then unique. For such a partition, write $\mathcal P \preceq \{\mathcal C_\ell\}_{\ell \in [1,k]}$ and define $c(\mathcal S_m)$ to be the value $\ell$ with $\mathcal S_m \subset \mathcal C_\ell$. Given a partition $\mathcal P$ refining $\{\mathcal C_\ell\}_{\ell \in [1,k]}$, let
\begin{equation*}
\Delta_{\mathcal P} = \prod_{m=1}^M \prod_{\substack{i < j \\ i,j \in \mathcal S_m}} \d_{ij}.
\end{equation*}
Every equivalence class of $\d$-products contains a unique $\Delta_{\mathcal P}$, where the condition that $\mathcal P$ refines $\{\mathcal C_\ell\}_{\ell \in [1,k]}$ corresponds precisely to the fact that we are only considering $\d_{ij}$ when $c_i \equiv c_j \mod r$. Equivalence classes of $\d$-products are thus in bijection with partitions of $[1,k]$ that refine $\{\mathcal C_\ell\}_{\ell \in [1,k]}$. For a partition $\mathcal P$, put
\begin{equation*}
w(\mathcal P) =\sum_{\Delta \sim \Delta_{\mathcal P}} (-1)^{|\Delta|},
\end{equation*}
so that
\begin{equation*}
\prod_{\ell = 1}^r \prod_{\substack{i,j \in \mathcal C_{\ell} \\ i < j}} (1-\d_{ij}) = \prod_{\mathcal P \preceq \{\mathcal C_\ell\}_{\ell \in [1,k]}} w(\mathcal P) \Delta_{\mathcal P},
\end{equation*}
and the sum over $a_i$'s in \eqref{eq:mvarb:rkstartingpoint} is equal to
\begin{equation*}
\subsum{\mathcal P \preceq \{\mathcal C_\ell\}_{\ell \in [1,k]} \\ \mathcal P = \{\mathcal S_1, \dots, \mathcal S_M\}} w(\mathcal P) \subsum{a_1, \dots, a_k \\ 1 \le a_i \le q_i \\ (a_i,q_i) = 1 \\ \sum_i a_i/q_i \in \Z} \prod_{m=1}^M E_{r,c(\mathcal S_m)} \left(\sum_{i \in \mathcal S_m} \frac{a_i}{q_i}\right).
\end{equation*}

By the same reasoning as in \cite{MontgomerySoundararajanPrimesIntervals}, the contribution to $R_k(h;r,c_1, \dots, c_k)$ from terms where $|\mathcal S_m| \ge 3$ for some $m$ is $O_r(h^{(k-1)/2 + \ep})$, so that
\begin{align}
R_k(h;r,c_1, \dots, c_k) = \subsum{\mathcal P \preceq \{\mathcal C_\ell\}_{\ell \in [1,k]} \\ \mathcal P = \{\mathcal S_1, \dots, \mathcal S_M\} \\ |\mathcal S_m| \le 2 \:\forall m} &w(\mathcal P) \subsum{q_1, \dots, q_k \\ 1 < q_i|Q} \prod_{i=1}^k \frac{\mu(q_i)}{\phi(q_i)} \nonumber \\
&\cdot\subsum{a_1, \dots, a_k \\ 1 \le a_i \le q_i \\ (a_i,q_i) = 1 \\ \sum_i a_i/q_i \in \Z} \prod_{m=1}^M  E_{r,c(\mathcal S_m)}\left(\sum_{i\in \mathcal S_m}\frac{a_i}{q_i}\right) + O_{r,k}(h^{(k-1)/2 + \ep}).\label{eq:Rkc1ck:midexpansion}
\end{align}
Suppose that $\mathcal P$ consists of $j$ doubleton sets $\mathcal S_1, \dots, \mathcal S_j$ and $k-2j$ singleton sets $\mathcal S_{j+1}, \dots, \mathcal S_{k-j}$. Note that the number of these partitions depends on the partition $\{\mathcal C_\ell\}$, because of the constraint that $\mathcal P \preceq \{\mathcal C_\ell\}_{\ell \in [1,k]}$. The term in $R_k(h;r,c_1, \dots, c_k)$ corresponding to a fixed such partition $\mathcal P$ is
\begin{equation}\label{eq:mvarb:fixPtermsdoublessingles}
(-1)^j  \subsum{q_1, \dots, q_k \\ 1 < q_i|Q} \prod_{i=1}^k \frac{\mu(q_i)}{\phi(q_i)} \subsum{a_1, \dots, a_k \\ 1 \le a_i \le q_i \\ (a_i,q_i) = 1 \\ \sum_i a_i/q_i \in \Z} \prod_{m=1}^j  E_{r,c(\mathcal S_m)}\left(\sum_{i\in \mathcal S_m}\frac{a_i}{q_i}\right) \prod_{m = j+1}^{k-j} E_{r,c(\mathcal S_m)} \left(\frac{a_{\mathcal S_m}}{q_{\mathcal S_m}}\right),
\end{equation}
where we are slightly abusing notation in the final product by identifying the singletons $\mathcal S_m$ with their unique element.

For $1 \le m \le j$, define $b_m$ and $s_m$ by the relations
\begin{equation*}
\frac{b_m}{s_m} = \sum_{i \in \mathcal S_m} \frac{a_i}{q_i} \mod 1, \qquad 1 \le b_m \le s_m, \qquad (b_m,s_m) = 1,
\end{equation*}
and define
\begin{equation*}
H_m\left(\frac bs\right) = E_{r,c(\mathcal S_m)}\left(\frac bs\right) \subsum{d_1, d_2|Q \\ 1 < d_i} \frac{\mu(d_1)\mu(d_2)}{\phi(d_1)\phi(d_2)}\subsum{e_1, e_2 \\ 1 \le e_i \le d_i \\ (e_i,d_i) = 1 \\ \tfrac{e_1}{d_1} + \tfrac{e_2}{d_2} = \tfrac bs \mod 1}1.
\end{equation*}
Then \eqref{eq:mvarb:fixPtermsdoublessingles} is equal to
\begin{equation}
\subsum{s_1, \dots, s_j \\ s_i|Q} \subsum{b_1, \dots, b_j \\ 1 \le b_i \le s_i \\ (b_i,s_i) = 1} \prod_{i=1}^j H_i\left(\frac{b_i}{s_i}\right) \subsum{q_{j+1}, \dots, q_{k-j} \\ 1 < q_i|Q} \subsum{a_{j+1}, \dots, a_{k-j} \\ 1 \le a_i \le q_i \\ (a_i,q_i) = 1 \\ \sum_i a_i/q_i + \sum_i b_i/s_i \in \Z} \prod_{i=j+1}^{k-j} \frac{\mu(q_i)}{\phi(q_i)} E_{r,c(\mathcal S_i)} \left(\frac{a_i}{q_i}\right).\label{eq:mvarb:fixPtermsdoublessingles:part2}
\end{equation}
Now separate the indices $i$ with $s_i = 1$. To do so, let $\mathcal L = \{i: s_i > 1\}$. We can again rewrite \eqref{eq:mvarb:fixPtermsdoublessingles:part2} as
\begin{equation}\label{eq:mvarb:sumoverL}
\sum_{\mathcal L \subset [1,j]} M(\mathcal L)\prod_{i \not\in \mathcal L} H_i(1), 
\end{equation}
where
\begin{equation*}
M(\mathcal L) = \subsum{(s_i)_{i \in \mathcal L} \\ 1 < s_i|Q} \subsum{(b_i)_{i \in \mathcal L} \\ 1 \le b_i \le s_i \\ (b_i,s_i) = 1} \prod_{i\in \mathcal L} H_i\left(\frac{b_i}{s_i}\right) \subsum{q_{j+1}, \dots, q_{k-j} \\ 1 < q_i|Q} \subsum{a_{j+1}, \dots, a_{k-j} \\ 1 \le a_i \le q_i \\ (a_i,q_i) = 1 \\ \sum_i a_i/q_i + \sum_i b_i/s_i \in \Z} \prod_{i=j+1}^{k-j} \frac{\mu(q_i)}{\phi(q_i)} E_{r,c(\mathcal S_i)} \left(\frac{a_i}{q_i}\right).
\end{equation*}
Note that $M(\emptyset) = V_{2k-j}(Q,h;r,c(\mathcal S_{j+1}), \dots, c(\mathcal S_{k-j}))$.

By precisely the same arguments as in \cite{MontgomerySoundararajanPrimesIntervals}, the contributions when $|\mathcal L| \ge 1$ can be absorbed into the error term. Moreover, 
\begin{align*}
H_i(1) &= E_{r,c(\mathcal S_i)}(1) \subsum{d|Q \\ d > 1} \frac{\mu(d)^2}{\phi(d)} \\
&= \left(\frac hr + O(1)\right)  \subsum{d|Q \\ d > 1} \frac{\mu(d)^2}{\phi(d)}.
\end{align*}
Thus the expression \eqref{eq:mvarb:sumoverL} is equal to
\begin{equation*}
\left(\frac hr \subsum{d|Q \\ d>1} \frac{\mu(d)^2}{\phi(d)}\right)^j  V_{k-2j}(Q,h;r,c(\mathcal S_{j+1}), \dots, c(\mathcal S_{k-j})) + O_{r,k}(h^{(k-1)/2 + \ep}).
\end{equation*}
Inserting this back into \eqref{eq:Rkc1ck:midexpansion} yields
\begin{align*}
&R_k(Q,h;r,c_1, \dots, c_k) = \sum_{0 \le j \le k/2} (-1)^j \nonumber \\
&\subsum{\mathcal P \preceq \{\mathcal C_\ell\}_{\ell \in [1,k]} \\ \mathcal P = \{\mathcal S_1, \dots, \mathcal S_M\} \\ |\mathcal S_m| \le 2 \:\forall m \\ |\mathcal P| = k-j} \left(\frac hr \subsum{d|Q \\ d>1} \frac{\mu(d)^2}{\phi(d)}\right)^j V_{k-2j}(Q,h;r,c(\mathcal S_{j+1}), \dots, c(\mathcal S_{k-j})) + O_{r,k}(h^{(k-1)/2 + \ep}).
\end{align*}

We are finally prepared to appeal to Theorem \ref{thm:mvarb:msthm1foraps}. If $k$ is odd, then so is $k-2j$, so there is no main term. Suppose that $k$ is even. Recall that $\mathcal B(j+1, \dots, k-j)$ denotes the set of perfect matchings of the set $\{j+1, \dots, k-j\}$. Then the main term is
\begin{align}
&\sum_{0 \le j \le k/2} (-1)^j \subsum{\mathcal P \preceq \{\mathcal C_\ell\}_{\ell \in [1,k]} \\ \mathcal P = \{\mathcal S_1, \dots, \mathcal S_M\} \\ |\mathcal S_m| \le 2 \:\forall m \\ |\mathcal P| = k-j} \left(\frac hr \subsum{d|Q \\ d>1} \frac{\mu(d)^2}{\phi(d)}\right)^j \sum_{\sigma \in \mathcal B(j+1, \dots, k-j)} \prod_{(i_1,i_2) \in \sigma} V_2(Q,h; r,c(\mathcal S_{i_1}), c(\mathcal S_{i_2})),
\end{align}
which proves the first claim in Theorem \ref{thm:mvarb:msthm2foraps}.

By Lemma \ref{lem:mvarb:V2comp}, $V_2(Q,h;r,c(\mathcal S_{i_1}), c(\mathcal S_{i_2})) = O_{r,k}(h)$ unless $c(\mathcal S_{i_1}) = c(\mathcal S_{i_2})$. So, the largest term comes from those $\sigma$ with $c(\mathcal S_{i_1}) = c(\mathcal S_{i_2})$ for all $(i_1, i_2) \in \sigma$. Note that the error term is then quite large; it is only smaller by a factor of $(\log h)^{-1}$. 

If there exists some $\sigma$ with $c(\mathcal S_{i_1}) = c(\mathcal S_{i_2})$ for all $(i_1, i_2) \in \sigma$, then it must be that $|\mathcal C_{\ell}|$ is even for all $\ell$. Moreover, each term in this sum corresponds to a perfect pairing of $[1,k]$ such that for each pair $(i_1,i_2)$, $c_{i_1} = c_{i_2}$; either two indices are paired by lying in the same $\mathcal S_m$, or by lying in the same element of $\sigma$. The choice of $\mathcal P$ then corresponds to choosing $j$ of these pairs. Note also that $V_2(Q,h;r,c,c) = V_2(Q,h;r,c',c') + O_r(h^{1/2+\ep})$ for any $c, c'\mod r$, which allows us to simplify the main term in this case to get
\begin{equation*}
\#\tilde{\mathcal B}(c_1, \dots, c_k) \sum_{0 \le j \le k/2}(-1)^j \binom{k/2}{j} \left(\frac hr \subsum{d|Q \\ d>1} \frac{\mu(d)^2}{\phi(d)}\right)^j  V_2(Q,h;r,0,0)^{k/2-j},
\end{equation*}
where $\#\tilde{\mathcal B}(c_1, \dots, c_k)$ is the number of ways to pair the $c_i$'s such that every pair has equal values. By the binomial theorem, this is
\begin{equation*}
\#\tilde{\mathcal B}(c_1, \dots, c_k) \left(V_2(Q,h;r,0,0) - \frac hr \subsum{d|Q\\d>1} \frac{\mu(d)^2}{\phi(d)}\right)^{k/2}.
\end{equation*}
By Lemma \ref{lem:mvarb:V2comp}, this is
\begin{equation*}
\#\tilde{\mathcal B}(c_1, \dots, c_k) \left( - h \frac{\phi(r)}{r} \log h  + C_0(r)h\right)^{k/2} + O_{r,k}(h^{(k-1)/2 + \ep}),
\end{equation*}
for $C_0(r)$ defined in \eqref{eq:mvarb:C0rdef}, which gives the result.

\section{Weighting by smooth functions}\label{sec:smooth-functions}

We now consider sums of singular series weighted by smooth functions and the proofs of Theorems \ref{thm:msthm1-for-smooth} and \ref{thm:msthm2-for-smooth}. Theorem \ref{thm:msthm1-for-smooth} follows arguments identical to those in the proof of Theorem \ref{thm:mvarb:msthm1foraps} as well as Theorem 1 of \cite{MontgomerySoundararajanPrimesIntervals}. In particular, all estimates used in bounding $E(\a)$ in the proof of Theorem 1 of \cite{MontgomerySoundararajanPrimesIntervals} hold for the sums $E_{f_i,h}(\a)$ that we consider in the smooth setting, and the remainder of the proof is identical.

Accordingly, for the proof of Theorem \ref{thm:msthm1-for-smooth} we restrict our attention to relevant estimates of the sums $E_{f_i,h}(\a)$, which is the only place where the proof differs. These estimates, the equivalents of Lemmas 4 and 6 from \cite{MontgomeryVaughanReducedResidues}, are contained in Section \ref{subsec:exponentialsums-smooth}.

Similarly for Theorem \ref{thm:msthm2-for-smooth}, the proof is identical to that of Theorem \ref{thm:mvarb:msthm2foraps} presented in Section \ref{sec:mvarb:msthm2pf}, so we omit it. In Section \ref{subsec:pf-of-lem-V2-smooth}, we prove Lemma \ref{lem:V2-smooth}, whose proof follows similar lines as the proofs in Section \ref{sec:twoterms}.

\subsection{Exponential sums weighted by smooth functions} \label{subsec:exponentialsums-smooth}

\begin{lemma}\label{lem:mv-lemma4-smooth}
Let $m,h \ge 1$, let $f_1,f_2: \R \to \R$ be a smooth functions with compact support such that $|\hat{f_i}(\xi)| = O(|\xi|^{-2})$, and define $E_{f_i,h}$ by \eqref{eq:mvaw:defEaphisum}. Then 
\begin{equation}\label{eq:lem4_efh_noalphashift}
\subsum{\mu \mod m \\ \mu \ne 0} E_{f_1,h} \left(\frac{\mu}{m}\right)E_{f_2,h}\left(\frac{\mu}{m}\right) \ll_{f_1,f_2}  mh^{-2} \min\{m^3,h^3\}.
\end{equation}
Moreover, for any $\a \in \R$, 
\begin{equation}\label{eq:lem4_efh_alphashift}
\subsum{\mu \mod m} E_{f_1,h}\left(\frac{\mu}{m}+ \a\right)E_{f_2,h}\left(\frac{\mu}{m}+ \a\right) = h^2 \hat{f}\left(-\frac hm (\bar{m\a})\right)^2 + O_f(mh^{-2} \min\{m^3,h^3\}).
\end{equation}
\end{lemma}
\begin{proof}
Begin with the second statement. 
Expand via \eqref{eq:Efh_fourier} to get
\begin{align*}
\subsum{\mu \mod m} &E_{f_1,h}\left(\frac{\mu}{m}+\a\right)E_{f_2,h}\left(\frac{\mu}{m}+\a\right) \\
&= h^2 \subsum{\mu \mod m} \hat{f_1}(-h(\bar{\mu/m + \a}))\hat{f_2}(-h(\bar{\mu/m+\a})) \\
&\qquad + O\Big(mh^{-2} + \subsum{\mu \mod m} \hat f_1(-h(\bar{\mu/m + \a})) + \hat f_2(-h(\bar{\mu/m+\a}))\Big) \\
&= h^2 \subsum{\mu \mod m} \hat{f_1}(-h\bar{(\mu/m+ \a)})\hat{f_2}(-h\bar{(\mu/m+ \a)}) + O_{f_1,f_2}\left(\min\{m,m^2h^{-2}\}\right).
\end{align*}
Let $\mu_0$ be the value of $\mu$ such that $\left|\bar{\frac{\mu}{m} + \a}\right|$ is minimized; then $\bar{\frac{\mu}{m} + \a} = \bar{m\a}/m$. If $m \le h$, then 
\begin{align*}
h^2 \subsum{\mu \mod m} \hat f_1&(-h(\bar{\mu/m+ \a}))\hat f_2(-h(\bar{\mu/m+ \a})) \\
&= h^2 \hat f_1\left(-\frac hm \bar{m\a}\right)\hat f_2\left(-\frac hm \bar{m\a}\right) + h^2 \subsum{\mu \mod m \\ \mu \ne \mu_0} \hat f_1(-h(\bar{\mu/m+ \a}))\hat f_2(-h(\bar{\mu/m+ \a})) \\
&\ll_{f_1,f_2} h^2 + O(m^4h^{-2}),
\end{align*}
using the fact that $|\hat f_i(\xi)| = O(|\xi|^{-2})$. 
On the other hand, if $m > h$, then 
\begin{align*}
h^2 \subsum{\mu \mod m} \hat f_1&(-h(\bar{\mu/m + \a}))\hat f_2(-h(\bar{\mu/m + \a})) \\
&= h^2 \subsum{\mu \mod m \\ |\bar{\mu/m + \a}| \le m/h} \hat f_1(-h(\bar{\mu/m+\a}))\hat f_2(-h(\bar{\mu/m + \a})) \\
&\qquad + h^2 \subsum{\mu \mod m \\ |\bar{\mu/m + \a}| > m/h} \hat f_1(-h(\bar{\mu/m+\a}))\hat f_2(-h(\bar{\mu/m + \a})) \nonumber\\
&\ll_{f_1,f_2} h^2 \frac mh \hat f(0)^2 + h^2 \frac{m}{h} \nonumber \\
&\ll_{f_1,f_2} mh,
\end{align*}
as desired.

Equation \eqref{eq:lem4_efh_noalphashift} follows from the proof of \eqref{eq:lem4_efh_alphashift} for $\alpha = 0$ upon excluding the term $m=0$.
\end{proof}
\begin{remark}
In the case of analogous exponential sums in the function field case of polynomials over $\mathbf F_q[t]$, as explored in \cite{KuperbergOddMoments}, the sum analogous to $\sum_{\mu \mod m} E\left(\frac {\mu}{m} + \alpha\right)^2$ can be bounded by the analog of $mh$ \emph{whenever} $\alpha$ is large, i.e. whenever $\alpha$ is a rational function of sufficiently large degree (see Lemma 3.4 of \cite{KuperbergOddMoments}). The equivalent statement here would be that $\sum_{\mu \mod m} E\left(\frac {\mu}{m} + \alpha\right)^2$ is bounded by $mh$ whenever $\alpha$ is far away from an integer. However, this is not true: if $m$ is even and $\alpha$ is very close to $\frac 12$, say, then this sum will still have a contribution of size $h^2$.

In the function field case of \cite{KuperbergOddMoments}, the simplified and stronger bounds correspondingly yield a simplified proof of the analog of Theorem \ref{thm:msthm1-for-smooth}. The smooth weights $f_i$ make the exponential sums here cleaner, but because the function field-style bounds are not available for the sums in Lemma \ref{lem:mv-lemma4-smooth}, the simplified proof of the analog of Theorem \ref{thm:msthm1-for-smooth} in the function field case also fails to apply.
\end{remark}

The following lemma corresponds to Lemma 6 of Montgomery and Vaughan's work. 
\begin{lemma}\label{lem:mv-lemma6-smooth}
Let $f_1,f_2: \R \to \R$ be smooth functions with compact support such that $|\hat{f_i}(\xi)| = O(|\xi|^{-2})$, and define $E_{f_i,h}$ by \eqref{eq:mvaw:defEaphisum}. Fix $\a_1, \a_2 \in \R$. Then 
 \begin{equation}
\subsum{\mu \mod m} E_{f_1,h}\left(\frac{\mu}{m}+ \a_1\right)E_{f_2,h}\left(\frac{\mu}{m} + \a_2\right) \ll (m+h)E_{f_1\bar{f_2}}(\a_1-\a_2) + O(m).
\end{equation}
\end{lemma}
\begin{proof}
By Lemma 3 from \cite{MontgomeryVaughanReducedResidues},
\begin{align*}
\sum_{\mu \mod m} &E_{f_1,h}\left(\frac{\mu}{m}+\a_1\right)E_{f_2,h}\left(\frac{\mu}{m}+\a_2\right) \\
&= \sum_{\mu \mod m} h^2\hat{f_1}\left(-h\frac{\mu}{m}-h\a_1\right)\hat{f_2}\left(-h\frac{\mu}{m} -h\a_2\right) + O(m) \\
&\ll (m+h)\int_{-\infty}^{\infty} \hat{f_1}(-ht-h\a_1)\hat{f_2}(-ht-h\a_2)h^2 \dt  + O(m),
\end{align*}
keeping in mind that $f_i(x) \asymp f_i(y)$ whenever $|x-y| \le 1/h$. The integral is the convolution of $\hat{f_1}$ and $\hat{f_2}$:
\begin{align*}
(m+h)\int_{-\infty}^{\infty} \hat{f_1}(-ht-h\a_1)\hat{f_2}(-ht-h\a_2)h^2 \dt &= (m+h)h\int_{-\infty}^{\infty} \hat{f_1}(u)\hat{f_2}(u+h\a_1-h\a_2) \du \\
&= (m+h)h\hat{f_1}\ast\hat{\bar{f_2}}(h\a_2-h\a_1)k \\
&= (m+h)E_{f_1\bar{f_2},h}(\a_1-\a_2) + O(m/h),
\end{align*}
which yields the desired bound.
\end{proof}
\begin{remark}
The bound in Lemma \ref{lem:mv-lemma6-smooth} improves on the analogous lemma in \cite{MontgomeryVaughanReducedResidues} by a factor of $\log h$. However, the improvement in this lemma does not affect the error term in Theorem \ref{thm:msthm1-for-smooth}.
\end{remark}

\subsection{Proof of Lemma \ref{lem:V2-smooth}} \label{subsec:pf-of-lem-V2-smooth}

\begin{lemma}\label{lem:Sfh-smooth}
Fix $h \ge 1$, and let $f:\R_{\ge 0} \to \C$ be a smooth function with compact support $\mathrm{supp}(f) \subset (0,\infty)$ and such that $|\hat{f}(\xi)| \ll O(|\xi|^{-2})$. Define
\begin{equation*}
S(f,h) := \sum_{m=1}^\infty f\left(\frac mh\right) \mathfrak S(\{0,m\}).
\end{equation*}
Then 
\begin{equation*}
S(f,h) = \{\mathcal M f\}(2)h - \frac{\{\mathcal M f'\}(1)}{2} \log h + \frac{\{\mathcal M f'\}(1)}{2}\left( \gamma_0 - \log 2\pi - \frac{\{\mathcal M f''\}}{\{\mathcal M f'\}(1)}\right) + O_f(h^{-1/2 + \ep}),
\end{equation*}
where $\gamma_0$ denotes the Euler--Mascheroni constant and for a function $g$, $\{\mathcal M g\}(s)$ is the Mellin transform of $g$ defined in \eqref{eq:mellin-transform-definition}.
\end{lemma}
\begin{proof}
The proof proceeds almost entirely along the lines of Proposition 2.1 in \cite{lemkeoliversound}, so we will be brief. Define for $\mathrm{Re}(s) > 1$
\begin{equation*}
F(s) := \sum_{n \ge 1} \frac{\mathfrak S(\{0,n\})}{n^s},
\end{equation*}
so that
\begin{equation}\label{eq:Sfh-lem-mellin-transform}
S(f,h) = \frac 1{2\pi i} \int_{(2)} F(s) h^s \{\mathcal M f\}(s) \ds.
\end{equation}

As noted in \cite{lemkeoliversound} and \cite{MontgomerySoundararajanPrimesIntervals}, $F(s)$ admits a meromorphic continuation to $\mathrm{Re}(s) > -1/2$ via 
\begin{equation*}
F(s) = \zeta(s)\zeta(s+1)\prod_{p \text{ prime}} \left(1-\frac{1-1/p^s)^2}{(p-1)^2}\right).
\end{equation*}
Since $f$ is smooth and has compact support, $\{\mathcal M f\}(s)$ is analytic in $\mathrm{Re}(s) > 0$. It can be extended meromorphically to the complex plane via the identity
\begin{equation*}
\{\mathcal M f\} = -\frac 1s \{\mathcal M f'\},
\end{equation*}
which follows from integration by parts. Thus $\{\mathcal M f\}(s)$ has simple poles at all nonpositive integers and no other poles.

The result follows from moving the line of integration in \eqref{eq:Sfh-lem-mellin-transform} to $\mathrm{Re}(s) = -1/2 + \ep$ and recording the contributions from the simple pole at $s = 1$ and the double pole at $s = 0$.
\end{proof}

We are now ready to prove Lemma \ref{lem:V2-smooth}.

\begin{proof}
Begin by expanding 
\begin{align*}
V_2&(Q,h;f_1,f_2) \\
&= -E_{f_1,h}(1)E_{f_2,h}(1) + \sum_{q|Q}\frac{\mu(q)^2}{\phi(q)^2} \subsum{1 \le a \le q \\ (a,q) = 1} E_{f_1,h}\left(\frac aq \right) E_{f_2,h} \left(-\frac aq \right) \\
&= -h^2\hat{f_1}(0)\hat{f_2}(0) + O_{f_1,f_2}(1) \\
&\qquad + \sum_{q|Q} \frac{\mu(q)^2}{\phi(q)^2} \subsum{1 \le a \le q \\ (a,q) = 1} \sum_{m_1,m_2 = 1}^\infty f_1\left(\frac{m_1}{h}\right) f_2\left(-\frac{m_2}{h} \right) e\left((m_1-m_2)\frac aq \right) \\
&= -h^2\hat{f_1}(0) \hat{f_2}(0) + O_{f_1,f_2}(1) + \sum_{m_1,m_2=1}^\infty f_1\left(\frac{m_1}{h}\right)f_2\left(-\frac{m_2}h\right) \sum_{q|Q} \frac{\mu(q)^2}{\phi(q)^2} \mathbf c_q(m_1-m-2), \\
\end{align*}
where $\mathbf c_q(m)$ is a Ramanujan sum. Just as for the arithmetic progressions case, this simplifies to
\begin{align*}
&= -h^2\hat{f_1}(0) \hat{f_2}(0) + \subsum{m_1,m_2 = 1 \\ m_1 \ne m_2}^\infty f_1\left(\frac{m_1}{h}\right) f_2\left(-\frac{m_2}{h}\right) \mathfrak S(\{0,m_1-m_2\}) \\
&\qquad + \sum_{m=1}^\infty f_1\left(\frac{m}{h}\right)f_2\left(-\frac mh\right)+ O_{f_1,f_2}(1),
\end{align*}
since for any $m \ne 0$, by our choice of $Q$,
\begin{equation*}
\sum_{q|Q} \frac{\mu(q)^2}{\phi(q)^2} \mathbf c_q(m) = \prod_{p|m} \left(1 + \frac 1{p-1}\right)\prod_{p\nmid m} \left(1-\frac 1{(p-1)^2}\right) = \mathfrak S(\{0,m\}) + O(h^{-2}). 
\end{equation*}

The sum over $m$ can be interpreted as a Riemann sum as $h \to \infty$, yielding
\begin{align*}
\sum_{m=1}^\infty f_1\left(\frac{m}{h}\right) f_2\left(-\frac mh\right) &= h \int_0^\infty f_1(x) f_2(-x) \dx + O_{f_1,f_2}(1) \\
&= h (f_1 \ast f_2)(0) + O_{f_1,f_2}(1).
\end{align*}
Similarly, the sum over $m_1$ and $m_2$ can also be interpreted as a Riemann integral (and also as the convolution of $f_1$ and $f_2$), yielding
\begin{align*}
\subsum{m_1,m_2 = 1 \\ m_1 \ne m_2}^\infty &f_1\left(\frac{m_1}{h}\right) f_2\left(-\frac{m_2}{h}\right) \mathfrak S(\{0,m_1-m_2\}) \\
&= \sum_{m=1}^\infty \mathfrak S(\{0,m\}) \sum_{n=1}^\infty \left(f_1\left(\frac nh\right) f_2 \left(\frac{m-n}{h}\right) + f_1\left(\frac nh\right)f_2\left(\frac{-m-n}{h}\right)\right) \\
&= \sum_{m=1}^\infty \mathfrak S(\{0,m\}) h\left((f_1\ast f_2)\left(\frac mh\right) \right) + O_{f_1,f_2}\left(\sum_{m=1}^{Ch} \mathfrak S(\{0,m\})\right),
\end{align*}
where $C$ is a constant large enough that $f_1(x) = f_2(x) = 0$ for any $|x| \ge C/2$; note that $C$ depends only on $f_1$ and $f_2$. By the results of \cite{gallaghershortintervals} and \cite{MontgomerySoundararajanPrimesIntervals}, the error term is $O_{f_1,f_2}(h)$. For the inside sum, apply Lemma \ref{lem:Sfh-smooth} with $f = f_1\ast f_2$ to get
\begin{align*}
\subsum{m_1,m_2 = 1 \\ m_1 \ne m_2}^\infty &f_1\left(\frac{m_1}{h}\right) f_2\left(-\frac{m_2}{h}\right) \mathfrak S(\{0,m_1-m_2\}) \\
&= \{\mathcal M f\}(2)h^2 - \frac{\{\mathcal M f'\}(1)}{2} h \log h + O_{f_1,f_2}(h),
\end{align*}
which, after collecting terms, implies the result.
\end{proof}

\bibliographystyle{amsplain}
\bibliography{singseries}

\end{document}